\documentclass[11pt]{article}
\usepackage{amsfonts}
\usepackage{amsmath}
\usepackage{amsthm}
\usepackage{amssymb}
\newcommand{\br}{\mathbb R}
\numberwithin{equation}{section}
\newtheorem{theorem}{Theorem}[section]
\newtheorem{lemma}{Lemma}[section]
\newtheorem{proposition}{Proposition}[section]
\newtheorem{corollary}{Corollary}[section]

\newtheorem{remark}{Remark}[section]

\newcommand{\tow}{\rightharpoonup}

\title{\bf  Convergence of Distributed Optimal Control Problems
 Governed by Elliptic Variational Inequalities}
\author{\
Mahdi Boukrouche\thanks{PRES Lyon University, University of
 Saint-Etienne, Laboratory of Mathematics, LaMUSE EA-3989,
23 rue Paul Michelon, 42023 Saint-Etienne, France.
  E-mail: Mahdi.Boukrouche@univ-st-etienne.fr} \and and \and
Domingo A. Tarzia \thanks{Corresponding author: Departamento de
Matem\'atica-CONICET, FCE, Univ. Austral, Paraguay 1950, S2000FZF
Rosario, Argentina. Tel: +54-341-5223093, Fax:  +54-341-5223001.
E-mail: DTarzia@austral.edu.ar} }

\begin{document}
\maketitle

\date{}
\begin{abstract}
First, let $u_{g}$ be the unique solution of an  elliptic
 variational inequality with source term  $g$.
We establish, in the general case, the error estimate  between
$u_{3}(\mu)=\mu u_{g_{1}}+ (1-\mu)u_{g_{2}}$
and $u_{4}(\mu)=u_{\mu g_{1}+ (1-\mu ) g_{2}}$
for $\mu\in [0 , 1]$. Secondly, we consider a family of distributed
optimal control problems governed by elliptic variational
inequalities over the internal energy $g$ for each positive heat
transfer coefficient $h$ given on a part of the boundary of the
domain. For a given  cost functional and using some monotony
property between $u_{3}(\mu)$ and $u_{4}(\mu)$ given in F. Mignot,
J. Funct. Anal., 22 (1976), 130-185, we prove the strong convergence
of the optimal controls and states associated to this family of
distributed optimal control problems governed by elliptic
variational inequalities to a limit Dirichlet distributed optimal
control problem, governed also by an elliptic variational
inequality, when the parameter $h$ goes to infinity. We obtain this
convergence without using the adjoint state problem (or the Mignot's
conical differentiability) which is a great advantage with respect
to the proof given in C.M. Gariboldi - D.A. Tarzia, Appl. Math.
Optim., 47 (2003), 213-230, for optimal control problems governed by
elliptic variational equalities.

\end{abstract}

\smallskip
\begin{center}
{ Short title : Convergence of optimal controls
         for obstacle problems}
\end{center}

\maketitle
\centerline{ Accepted in Comput Optim Appl
DOI 10.1007/s10589-011-9438-7}

\section{Introduction}\label{intro}
Let $V$  a Hilbert space, $V'$  its topological dual,  $K$ be  a
closed, convex and non empty set in $V$, $g$ in $V'$ and a bilinear
form $a : V\times V\to \mathbb{R}$, which is
 symmetric, continuous and coercive form that is there exists a constant
 $m>0$ such that $m\|v\|^{2}\leq a(v , v)$ for all $v$ in $V$. It is well known
\cite{Stamp64, JLL1, Kind80} that for each $g\in V'$
there exists a unique solution $u\in K$, such that
\begin{equation}\label{eq1}
a(u \, ,\, v-u) \geq  \langle g \,,\, v- u\rangle  \qquad \forall
v\in K,
\end{equation}
where $\langle \cdot , \cdot \rangle$ denotes the duality pairing
between $V$ and $V'$. So we can consider  $g \mapsto u= u_{g}$ as a
function from $V'$ to   $K$. Let $u_{i}= u_{g_{i}}$ be the
corresponding solution of (\ref{eq1}) with $g=g_{i}$ for $i=1 , 2$.
We define  for  $\mu\in [0 , 1]$
\begin{equation}\label{d}
u_{3}(\mu)= \mu u_{1}+ (1-\mu)u_{2}, \quad g_{3}(\mu)= \mu g_{1} + (1-\mu)g_{2},
\quad\mbox{ and }\quad u_{4}(\mu)= u_{g_{3}(\mu)}.
\end{equation}
In \cite{MB-DT1}, we  established the necessary and sufficient condition to obtain that  the convex combination
$u_{3}(\mu)$ is the unique solution of the elliptic variational inequality \eqref{eq1} with source term
 $g_{3}(\mu)$, namely
\begin{equation}\label{u3}
u_{4}(\mu) =  u_{3}(\mu) \quad \forall \mu\in [0 , 1]
\mbox{ if and only if } \alpha = \beta =0,
\end{equation}
with
\begin{gather}\label{alp}
\alpha=\alpha(g_{1}) := a(u_{1} ,  u_{2} - u_{1})
-\langle  g_{1} ,  u_{2} - u_{1}\rangle,\\
\label{bet} \beta=\beta(g_{2}) := a(u_{2} ,  u_{1} - u_{2}) -\langle
g_{2} ,  u_{1} - u_{2}\rangle.
\end{gather}

In Section \ref{sec1},  we establish the error estimate between
$u_{3}(\mu)$ and $u_{4}(\mu)$ in the case where  $\alpha$ and
$\beta$ defined by  (\ref{alp}) and (\ref{bet}) are not equal to
zero.  We obtain also some other information concerning
 $u_{3}(\mu)$ and $u_{4}(\mu)$ which will be used in Section \ref{ocpb}. We can not obtain,
for an arbitrary convex $K$, a needed monotony property of
 $u_{3}(\mu)$ and $u_{4}(\mu)$ that $u_{4}(\mu) \leq
u_{3}(\mu)$  $\forall \mu\in [0 , 1]$ \cite{Mingot1} but we can
obtain this inequality for the complementary free boundary problems
given in Section \ref{pobs}.

In Section \ref{pobs},  we consider  a family of free boundary
problems with mixed boundary conditions associated to particular
cases  of the elliptic variational inequality (\ref{eq1}). We study
some dependence properties of the solutions to this family of
elliptic variational inequalities, on the internal energy $g$ (see
more details in the complementary problem (\ref{pr1}) or the
variational inequalities (\ref{iv1}) or (\ref{iv2})) and also on the
heat transfer coefficient $h$ which is  characterized in  the Newton
law or the Robin boundary condition (\ref{bc2}) (see also the
variational inequality (\ref{iv2})). Note that mixed boundary
conditions play an important role in various applications
\cite{HMRS2009, TT}.

In Section \ref{ocpb}, first for a given constant $M>0$ we consider
$g$ as a control variable for the cost functional (\ref{e4.1}), then
we formulate the distributed optimal control problem associated to
the variational inequality (\ref{iv1}). We also formulate the family
of distributed optimal control problems associated to the
variational inequality of (\ref{iv2}), which depend on a positive
parameter $h$. With the above dependence properties obtained in
Section \ref{pobs}, the inequality obtained in Section \ref{sec1}
and by using the monotony property \cite{Mingot1} between
$u_{3}(\mu)$ and $u_{4}(\mu)$, we obtain a new proof of the strict
convexity of the cost functional which is not given in
\cite{Mingot1} and then the existence and the uniqueness of the
optimal control $g_{op}$ holds. We obtain similar results for the
optimal control $g_{op_{h}}$. We remark here that the strict
convexity of the cost functional is automatically true (then the
uniqueness of the optimal control problems holds) when the
equivalence (\ref{u3}) is verified.

Then, we prove that the optimal control $g_{op_{h}}$ and its
corresponding state $u_{g_{op_{h}h}}$ are strongly convergent to
$g_{op}$ and $u_{g_{op}}$ respectively, when $h\to +\infty$, in
adequate functional spaces. This asymptotic behavior can be
considered very important in the optimal control for heat transfer
problems because the Dirichlet boundary condition, given in
(\ref{bc1}) is not a relevant physical condition to impose on the
boundary; the true relevant physical condition is given by the
Newton law or the Robin boundary condition (\ref{bc2})
\cite{CaJa1959}. Therefore, the goal of this paper is to approximate
a Dirichlet optimal control problem, governed by an elliptic
variational inequality, by a Neumann optimal control problems,
governed by elliptic variational inequalities, for a large positive
coefficient $h$. Moreover, from a numerical analysis point of view
it maybe preferable to consider approximating Neumann problems in
all space $V$ (see the variational inequality (\ref{iv2})), with
parameter $h$, rather than a Dirichlet problem in a restriction of
the space $V$ (see the variational inequality (\ref{iv1})).

We note here that we do not need to consider the adjoint state for
problems (\ref{iv1}) and (\ref{iv2}) as in \cite{GT, MT} in order to
prove the convergence when $h\to +\infty$. This is a very important
advantage of our proof with respect to the previous one given for
variational equalities in \cite{GT}. This fact was possible because
we do not need to use the cornerstone Mignot's conical
differentiability of the cost functional \cite{Mingot1}.

Different problems with distributed optimal control governed by
partial differential equations can be found in the following books
\cite{Barbu84a, JLL, NPS2006, T2010}. Moreover, we describe briefly
some works on optimal control governed by elliptic variational
inequalities, see for example:
 \cite{A2006,MP1984} on optimality conditions for the penalized problem,
 \cite{B1997b} on augmented Lagrangian algorithms,
 \cite{Mber2, BM2000, IK2000, K2008} on Lagrange multipliers,
 \cite{YC2004} on quasilinear elliptic variational inequalities,
 \cite{H2008} on estimation of a parameter involved in a variational inequality model,
 \cite{Acapatina2000} on  optimal control problems of variational inequalities for Signorini problem,
 \cite{Pa1977} on  optimal control for variational inequalities governed by a pseudomonotone operator,
 \cite{Jhas1986} when optimal control problem for a variational
inequality is approximated by a family of finite-dimensional
problems, \cite{H2001} on the identification of a distributed
parameter, and \cite{MRT2006} on regularization techniques with
state constraints. In conclusion, many practical applications
ranging from physical and engineering sciences to mathematical
finance are modeled properly by elliptic and parabolic variational
inequalities (see \cite{H2008, H2009, IK2008} and their references
within them).

\section{Some general results}\label{sec1}

In \cite{MB-DT1} we proved the  equivalence (\ref{u3}).
In order to study optimal control problems in Section  \ref{ocpb}
 it is useful for us, to obtain the error estimate between
 $u_{3}(\mu)$ and $u_{4}(\mu)$ when the equivalence (\ref{u3}) is not
satisfied.

\begin{theorem}\label{th1}
Let $u_{1}$ and $u_{2}$ be the two  solutions of the variational inequality (\ref{eq1})
with respectively  as source term  $g_{1}$ and $g_{2}$, then
 we have the following estimate
\begin{eqnarray*}\label{eq3.1}
m\|u_{4}(\mu) -u_{3}(\mu)\|^{2}_{V} + \mu I_{14}(\mu) + (1-\mu) I_{24}(\mu)
\leq \mu(1-\mu)(\alpha + \beta), \quad \forall \mu\in [0  , 1]
\end{eqnarray*}
where $\alpha$ and $\beta$ are  defined by  {\rm (\ref{alp})} and
{\rm (\ref{bet})} respectively and
$$I_{14}(\mu) =  a(u_{1}\, , \, u_{4}(\mu) - u_{1})
-\langle g_{1} , u_{4}(\mu) - u_{1}\rangle \geq 0$$
$$I_{24}(\mu) = a(u_{2}\, , \, u_{4}(\mu) - u_{2} )
-\langle g_{2} , u_{4}(\mu) - u_{2}\rangle\geq 0.$$
\end{theorem}

\begin{proof}
As $u_{4}(\mu)$ is the unique solution of the variational inequality
\begin{eqnarray*}
  a(u_{4}(\mu) \, , \, v-u_{4}(\mu))
- \langle g_{3}(\mu) , v- u_{4}(\mu)\rangle \geq 0, \quad \forall v\in K
\end{eqnarray*}
and $u_{3}(\mu) \in K$ so taking $v= u_{3}(\mu)$ in this
variational inequality, we have
\begin{eqnarray*}
m\|u_{4}(\mu)- u_{3}(\mu)\|_{V}^{2}\leq
a(u_{3}(\mu) \, , \, u_{3}(\mu)-u_{4}(\mu))-
\langle g_{3}(\mu) \, ,\,  u_{3}(\mu)- u_{4}(\mu)\rangle.
\end{eqnarray*}
Using  that $u_{3}(\mu)=\mu(u_{1} -u_{2})+  u_{2}$ and $g_{3}(\mu)=\mu(g_{1} -g_{2})+  g_{2}$ we obtain
\begin{eqnarray*}\label{eq3.2}
m\|u_{4}(\mu)- u_{3}(\mu)\|_{V}^{2}
&\leq&
\left[a(u_{2} \, , \,  u_{2}-u_{4}(\mu))
           - \langle g_{2} \, , \, u_{2}-u_{4}(\mu)\rangle\right]
\nonumber\\
&&+\mu\left[a( u_{2} \, , \, u_{1} -u_{2})
           - \langle g_{2} \, ,\, u_{1} -u_{2}\rangle\right]
\nonumber\\
&&+\mu^{2} \left[a(u_{1} -u_{2} \, , \, u_{1} -u_{2})
                 -\langle g_{1} -g_{2} \, , \, u_{1} -u_{2}\rangle\right]
\nonumber\\
&&+\mu \left[ a(u_{1} -u_{2} \, , \,  u_{2}-u_{4}(\mu))
             - \langle g_{1} -g_{2} \, , \,  u_{2}-u_{4}(\mu)\rangle \right]
\nonumber\\
&\leq&  -I_{24}(\mu) + \mu \beta - \mu^{2}\beta -\mu^{2}\alpha + \mu I_{24}(\mu)
\nonumber\\
&&
  + \mu \left[a(u_{1} \, , \,  u_{2}-u_{4}(\mu))
             - \langle g_{1} \, , \,  u_{2}-u_{4}(\mu)\rangle \right],
\end{eqnarray*}
so
\begin{eqnarray*}
m\|u_{4}(\mu)- u_{3}(\mu)\|_{V}^{2}
&\leq& \mu (1-\mu)(\alpha +\beta) -\left[\mu I_{14}(\mu)+(1-\mu)I_{24}(\mu) \right],
\end{eqnarray*}
which is the required result.
\end{proof}

The result of Theorem \ref{th1} will be used in Section \ref{ocpb}
(see Lemma \ref{l3}). Moreover, from Theorem \ref{th1} we deduce the
result obtained in \cite{MB-DT1} and more information concerning
$u_{3}(\mu)$ and $u_{4}(\mu)$ in the following corollary.
\begin{corollary}
\begin{eqnarray*}
\alpha(g_{1})=\beta(g_{2})= 0 \Longrightarrow \left\{
\begin{array}{ll}
(i)  & u_{3}(\mu) =u_{4}(\mu)  \qquad \forall \mu\in [0 , 1]\\ \\
(ii) & I_{14}(\mu)= I_{24}(\mu)=0 \qquad \forall \mu\in [0 , 1].
 \end{array}
\right.
\end{eqnarray*}
\end{corollary}

\begin{remark}\label{r1}
 We can not obtain a monotony property between $u_{3}(\mu)$ and $u_{4}(\mu)$
for a general variational inequality {\rm(\ref{eq1})}, precisely for any convex set $K$.
 But we can obtain it  when we consider the particular obstacle
 problems (see Section \ref{pobs}).
\end{remark}

\section{Dependence properties of solution of obstacle problem}\label{pobs}

Let $\Omega$ an open bounded set in $\br^{n}$ with its boundary
$\partial\Omega=\Gamma_{1}\cup\Gamma_{2}$. We suppose that
 $\Gamma_{1}\cap\Gamma_{2}=\emptyset$,
and  $mes(\Gamma_{1})>0$. We consider the following complementarity
problem:
\begin{equation}\label{pr1}
u\geq 0, \quad  u (-\Delta u - g) = 0, \quad -\Delta u - g\geq
0\quad a.e. \quad in\quad \Omega
 \end{equation}
\begin{equation}\label{bc1}
u= b \quad on \quad \Gamma_{1},\qquad \qquad
-{\partial u\over \partial n} = q \quad on \quad \Gamma_{2}
\end{equation}
and for a parameter $h>0$, we consider the complementarity problem
(\ref{pr1}) with the mixed boundary conditions :
\begin{equation}\label{bc2}
-{\partial u\over \partial n} = h(u-b) \quad on \quad \Gamma_{1}\qquad \qquad
-{\partial u\over \partial n} = q \quad on \quad \Gamma_{2}
\end{equation}
where $h$ is the heat transfer coefficient on $\Gamma_{1}$, $g$ is the internal energy, $b$ is the temperature on $\Gamma_{1}$, $q$ is the heat flux on $\Gamma_{2}$.

It is well known that the regularity of the mixed problem is
problematic in the neighborhood of  some part of the boundary, see
for example the book \cite{gri85}. A regularity for elliptic
problems with mixed boundary conditions is given in
\cite{bacuta2003, LCB2008}. Moreover, sufficient hypothesis on the
data in order to have the $H^{2}$ regularity for elliptic
variational inequalities are (\cite{R1987}, page 139):
\begin{equation}\label{hypo}
\partial \Omega\in C^{1, 1}, \quad g\in H=L^{2}(\Omega), \quad q\in
H^{3/2}(\Gamma_{2})
 \end{equation}
which are assumed from now on.

We define the spaces $V=H^{1}(\Omega)$, $V_{0}=\{ v\in V :\,
v_{|_{\Gamma_{1}}}=0\}$ and the convex sets given by
$$K=\{v\in V : \quad v|_{\Gamma_{1}}= b, \quad v\geq 0 \quad in \quad \Omega\},$$
$$K_{+}=\{v\in V : \quad  v\geq 0 \quad in \quad \Omega\}.$$

It is classical that, for a given positive $b\in H^{1\over 2}(\Gamma_{1})$,
 $q\in L^{2}(\Gamma_{2})$, and $g\in H$,
the two free boundary problems (\ref{pr1})-(\ref{bc1}) and
(\ref{pr1}), (\ref{bc2}) lead respectively to the following elliptic
variational problems: Find $u\in K$ such that
\begin{equation}\label{iv1}
 a(u , v-u) \geq ( g , v-u) - \int_{\Gamma_{2}}q(v-u)ds, \qquad \forall v\in K
\end{equation}
and find $u\in K_{+}$  such that
\begin{equation}\label{iv2}
  a_{h}(u \, ,\,  v-u) \geq ( g \, ,\,  v-u) - \int_{\Gamma_{2}}q(v-u)ds + h\int_{\Gamma_{1}}b(v-u)ds\quad \forall v\in K_{+}
\end{equation}
respectively, where
$$a(u  , v)= \int_{\Omega}\nabla u \nabla v dx,
\qquad (g , v) = \int_{\Omega} g  v dx,$$
$$a_{h}(u , v)= a(u , v)+ h\int_{\Gamma_{1}}u  v ds.$$

It is evident that {\rm\cite{Kind80}}
$$\exists  \lambda >0 \quad\mbox{ such that  }
\lambda \|v\|_{V}^{2}\leq a( v \, , \, v)\quad \forall v\in V_{0}.$$
Moreover {\rm\cite{TT, T}}
 $$ \exists\lambda_{1}>0 \quad\mbox{ such that  } \lambda_{h}\|v\|_{V}^{2}    \leq a_{h}(v , v)  \quad \forall v\in V,
\mbox{ with }  \lambda_{h}= \lambda_{1}\min\{1  \,, \, h\}$$
that is  $a_{h}$ is a bilinear continuous, symmetric and coercive
form on $V$, as $a$.

\begin{remark}
Note that we can easily obtain the same results of this paper for more general problem
than  {\rm(\ref{pr1})-(\ref{bc1})} and {\rm(\ref{pr1}), (\ref{bc2})} governed
by elliptic  variational inequalities under the assumption that the form $a$ must be
 bilinear, continuous  and coercive.
\end{remark}

\begin{remark}\label{r2.3}
The variational inequalities {\rm(\ref{iv1})} and {\rm (\ref{iv2})} are the  particular cases of {\rm (\ref{eq1})}  for the particular convex sets $K$ and $K_{+}$ and
\begin{equation}\label{L}
 <g \, , \, v > =  ( g , v) - \int_{\Gamma_{2}}q v ds,
\end{equation}
\begin{equation}\label{Lh}
 < g \, , \, v> =  ( g , v) - \int_{\Gamma_{2}}q v ds + h\int_{\Gamma_{1}}b v ds
\end{equation}
respectively.
Moreover for $g\geq 0$ in $\Omega$, $q\leq 0$ on $\Gamma_{2}$ and $b\geq 0$ on
$\Gamma_{1}$, then by the weak maximum principle, the unique solution of {\rm(\ref{iv1})}
is in $K$ and the unique solution of {\rm(\ref{iv2})} is in $K_{+}$ for each $h > 0$.
\end{remark}

For all $h>0$ and all $g\in H$, we associate $u=u_{g_{h}}$ the
unique solution of (\ref{iv2}) and $u=u_{g}$ the
unique solution of (\ref{iv1}).

\begin{lemma}\label{l2.3}
a) Let  $u_{g_{n}}$, $u_{g}$ two solutions of {\rm(\ref{iv1})} with
$g_{n}$ and $g$ in $H$ then we have
\begin{equation}\label{eq3.5}
  g_{n}\tow g \quad in \quad H \quad (weak)\quad as \quad n\to +\infty \quad then \quad  u_{g_{n}}\to u_{g}\quad in \quad V \quad
  (strong).
\end{equation}

Moreover, we have
\begin{equation}\label{eq3.6}
 g_{1}\geq g_{2} \quad in \quad \Omega \quad then \quad u_{g_{1}}\geq u_{g_{2}}\quad in \quad
 \Omega,
\end{equation}
\begin{equation}\label{eq3.7}
 u_{min(g_{1} , g_{2})} \leq u_{4}(\mu) \leq u_{max(g_{1} , g_{2})}, \qquad \forall \mu\in [0 ,
 1].
\end{equation}

b) Let  $u_{g_{n}h}$, $u_{gh}$ two solutions of {\rm(\ref{iv2})}
with $g_{n}$ and $g$ in $H$ and $h>0$ then we have
\begin{equation}\label{eq3.5h}
  g_{n}\tow g \quad in \quad H \quad (weak)\quad as \quad n\to +\infty \quad then \quad  u_{g_{n}h}\to u_{gh}\quad in \quad V \quad
  (strong).
\end{equation}

\end{lemma}

\begin{proof}
a) Let $g_{n}\tow g$ in $H$ as $n\to +\infty$, $u_{g_{n}}$ and
$u_{g}$ in $K$ such that
\begin{equation}\label{gh}
 a(u_{g_{n}} , v-u_{g_{n}}) \geq ( g_{n} , v-u_{g_{n}}) - \int_{\Gamma_{2}}q(v-u_{g_{n}})ds \qquad \forall v\in K.
\end{equation}
Set $z_{n}=u_{g_{n}}-B$ where $B\in K$ such that
$B|_{\Gamma_{1}}=b$, and taking $v=B$ in (\ref{gh})  we obtain the following inequalities
\begin{equation}\label{eq3.8}
\lambda\|z_{n}\|_{V}^{2} \leq a(z_{n},z_{n}) \leq -a(z_{n} , B) +
(g_{n},z_{n}) - \int_{\Gamma_{2}}q z_{n}ds.
\end{equation}

As  $g_{n}\tow g$ in $H$ then $\|g_{n}\|_{H}$ is bounded, then from
(\ref{eq3.8}) there exists a positive constant $C$  which do not
depend on $n$ such that $\|u_{g_{n}}\|_{V} \leq C$. Thus
\begin{equation}\label{eqW}
 \exists \eta \in V  \mbox{  such that }
 u_{g_{n}}\tow \eta    \mbox{  weakly in } V  \quad  \mbox{(strongly in } H),
\end{equation}
taking $n\to +\infty$ in (\ref{gh}), we get
\begin{equation}\label{qh}
 a(\eta , v-\eta ) \geq ( g , v-\eta ) - \int_{\Gamma_{2}}q(v-\eta )ds, \qquad \forall v\in K.
\end{equation}
By the uniqueness of the solution of (\ref{iv1}) we obtain that
$\eta= u_{g}$. Taking now $v= u_{g}$ in (\ref{gh}), and taking
$v=u_{g_{n}}$ in (\ref{iv1}) with $u= u_{g}$, then by addition we get
\begin{equation*}
a(u_{g_{n}} - u_{g} , u_{g_{n}}-u_{g} )\leq (g_{n} -g , u_{g_{n}}- u_{g}),
\end{equation*}
that is (\ref{eq3.5}).

\bigskip
Taking  in (\ref{iv1}) $v=u_{1}+(u_{1}-u_{2})^{-}$ (which is in $K$)  where $u=u_{1}$  and
$g= g_{1}$. Then taking in (\ref{iv1}) $v=u_{2} - (u_{1}-u_{2})^{-}$ (which also is in $K$)   where $u=u_{2}$
and $g= g_{2}$. By addition we get
$$a((u_{1}-u_{2})^{-} , (u_{1}-u_{2})^{-} ) \leq (g_{2}-g_{1} , (u_{1}-u_{2})^{-} )$$
so if $g_{2}-g_{1}\leq 0$ in $\Omega$ then $\|(u_{1}-u_{2})^{-}\|_{V}=0$, and as
$(u_{1}-u_{2})^{-}=0$ on $\Gamma_{1}$ we have  $u_{1}-u_{2}\geq 0$ in $\Omega$.
This gives (\ref{eq3.6}).
Finally (\ref{eq3.7}) follows from (\ref{eq3.6}) because
 $$min\{g_{1} , g_{2}\} \leq \mu g_{1} +(1-\mu) g_{2} \leq max\{g_{1} , g_{2}\}, \quad \forall \mu \in [0 , 1].$$

b) It is similar to a) for all $h>0$.
\end{proof}

\smallskip

Let now $g_{1}$,  $g_{2}$ in $H$, and  $u_{g_{1}h}$, $u_{g_{2}h}$
two solutions of the variational inequality \rm{(\ref{iv2})} with
$g=g_{1}$ and $g=g_{2}$ respectively, and the same $q$ and $h$. We
define also
$$u_{3h}(\mu)= \mu u_{g_{1}h} + (1-\mu)u_{g_{2}h}\quad \mbox{  and   }\quad  u_{4h}(\mu)= u_{(\mu g_{1} + (1-\mu)g_{2})h}.$$
So we obtain as in (\ref{eq3.7}) that
\begin{equation}\label{eq3.7h}
 u_{min(g_{1} , g_{2})h} \leq u_{4h}(\mu)
 \leq
 u_{max(g_{1} , g_{2})h}, \qquad \forall \mu\in [0 , 1].
\end{equation}

\begin{remark}
Taking $v=u^{+}$ in {(\ref{iv2})} we deduce that
$$a_{h}(u^{-} \, , \, u^{-} ) \leq -(g\, , \, u^{-} )
+ \int_{\Gamma{2}} q  u^{-} ds - h \int_{\Gamma{1}} b  u^{-} ds$$ so
for $h >0$ sufficiently large we can have  $u_{gh}\geq 0$ in
$\Omega$ with $g\leq 0$ in $\Omega$, for given $q\geq 0$ on
$\Gamma_{2}$ and $b\geq 0$ on $\Gamma_{1}$.
\end{remark}

 \begin{lemma}\label{l2.3h}
Let $g_{1}, g_{2}$ in $H$ and  $u_{g_{1}h}$, $u_{g_{2}h}$ two
solutions of the variational inequality {(\ref{iv2})} with  the same
$q$ and $h$. Suppose that $b$ is a positive constant and $q\geq 0$,
then we have
\begin{equation}\label{eq3.71}
g\leq 0 \mbox{ in } \Omega  \Longrightarrow u_{g_{h}}\leq b   \mbox{ in  } \Omega,  \mbox{ and } u_{g_{h}}\leq b   \mbox{ on   } \Gamma_{1},
\end{equation}
\begin{equation}\label{eq3.72}
 g_{2}\leq g_{1}\leq 0  \mbox{ in } \Omega, \quad and
\quad h_{2}\leq h_{1}
  \Longrightarrow  u_{g_{2}h_{2}}\leq u_{g_{1} h_{1}}   \mbox{ in  } \Omega,
\end{equation}
\begin{equation}\label{eq3.73}
  g \leq 0 \mbox{ in } \Omega
  \Longrightarrow u_{g_{h}}\leq  u_{g}   \mbox{ in  } \Omega, \quad \forall h>0.
\end{equation}
Moreover $\forall g\in H$, $\forall q\in L^{2}(\Gamma_{2})$ and $\forall b\in H^{1\over 2}(\Gamma_{1})$, we have
\begin{equation}\label{eq3.74}
 h_{2}\leq h_{1}
  \Longrightarrow \|u_{g_{h_{2}}}-  u_{g_{h_{1}}}\|_{V}\leq {\|\gamma_{0}||\over \lambda_{1}\min(1 , h_{2})}\|b-u_{g_{h_{1}}}\|_{L^{2}(\Gamma_{1})} (h_{1}-h_{2})
\end{equation}
where $\gamma_{0}$ is the  trace embedding from $V$ to
$L^{2}(\Gamma_{1})$ and $\|\gamma_{0}\|$ is its norm.

\end{lemma}
 \begin{proof}
Taking in (\ref{iv2}) $u= u_{g_{h}}$ and $v= u_{g_{h}}-(u_{g_{h}}-b)^{+}$ (which in $K_{+}$), we get
\begin{eqnarray*}\label{}
 - a_{h}(u_{g_{h}} \, ,\,  (u_{g_{h}}-b)^{+}) \geq -( g \, ,\, (u_{g_{h}}-b)^{+}) + \int_{\Gamma_{2}}q(u_{g_{h}}-b)^{+}ds - h\int_{\Gamma_{1}}b(u_{g_{h}}-b)^{+} ds,
\end{eqnarray*}

then
\begin{eqnarray*}
 a_{h}((u_{g_{h}}-b)^{+} \, ,\,  (u_{g_{h}}-b)^{+})
\leq  ( g \, ,\, (u_{g_{h}}-b)^{+})
- \int_{\Gamma_{2}}q(u_{g_{h}}-b)^{+}ds \leq 0,
\end{eqnarray*}
so (\ref{eq3.71}) holds.

\bigskip
To check  (\ref{eq3.72})
  we take first  in (\ref{iv2}) $v= u_{g_{1}h_{1}}+(u_{g_{2}h_{2}}-u_{g_{1}h_{1}})^{+}$,
  which is in $K_{+}$, where $u= u_{g_{1}h_{1}}$ is in $K_{+}$ with $g=g_{1}$
  and $h=h_{1}$,   and taking in (\ref{iv2})
  $v= u_{g_{2}h_{2}}-(u_{g_{2}h_{2}}-u_{g_{1}h_{1}})^{+}$, which is also in
  $K_{+}$,  where $u= u_{g_{2}h_{2}}$ is in $K_{+}$ with $g=g_{2}$  and
  $h=h_{2}$,
 then adding the two obtained inequalities we get
 \begin{eqnarray*}
  a_{h_{2}}  ( (u_{g_{2}h_{2}}-u_{g_{1}h_{1}})^{+}  \, ,\,
 (u_{g_{2}h_{2}}-u_{g_{1}h_{1}})^{+} )
 \leq ( g_{2} -g_{1}\, ,\, (u_{g_{2}h_{2}}-u_{g_{1}h_{1}})^{+} )ds
 \nonumber\\
 -
 (h_{2}- h_{1}) \int_{\Gamma_{1}}(u_{g_{1}h_{1}}-b) (u_{g_{2}h_{2}}-u_{g_{1}h_{1}})^{+}ds
 \end{eqnarray*}
and from (\ref{eq3.71}) we get (\ref{eq3.72}).

\smallskip
To check (\ref{eq3.73}), let $W= u_{g_{h}}- u_{g}$ and choose in (\ref{iv2})
$v= u_{g_{h}} -W^{+}$ which is in $K_{+}$, so
\begin{eqnarray}\label{e5.11}
a(u_{g_{h}} \, ,\, W^{+}) \leq ( g \, ,\, W^{+} ) - \int_{\Gamma_{2}}q W^{+}ds.
 \end{eqnarray}
 We choose, in (\ref{iv1}), $v= u_{g} +W^{+}$, which is in $K$ because from (\ref{eq3.71}),
 then we have $W^{+}= 0$ on $\Gamma_{1}$,  so
 \begin{eqnarray}\label{e5.12}
  a( u_{g} , W^{+}) \geq ( g \, ,\, W^{+} ) - \int_{\Gamma_{2}}q W^{+}ds.
 \end{eqnarray}
So from (\ref{e5.11}) and (\ref{e5.12}) we deduce that
  $a(W^{+} , W^{+}) \leq 0$. Then (\ref{eq3.73}) holds.

\smallskip
To finish the proof it remain to check  (\ref{eq3.74}). We choose
$v= u_{g_{h_{2}}}$ in (\ref{iv2}) where $u=u_{g_{h_{1}}}$, and  $v= u_{g_{h_{1}}}$ in (\ref{iv2}) where
$u=u_{g_{h_{2}}}$,  adding the two inequalities we get
\begin{eqnarray*}\label{e5.15}
\lambda_{1}\min\{1 , h_{2}\} \|u_{g_{h_{1}}} -
u_{g_{h_{2}}}\|_{V}^{2}
  &\leq&
 (h_{1} -h_{2})\|b-  u_{g_{h_{1}}}\|_{L^{2}(\Gamma_{1})} \|u_{g_{h_{1}}}-u_{g_{h_{2}}}\|_{L^{2}(\Gamma_{1})}
\nonumber\\
&\leq& \|\gamma_{0}\|(h_{1} -h_{2})\|b-
u_{g_{h_{1}}}\|_{L^{2}(\Gamma_{1})}
\|u_{g_{h_{1}}}-u_{g_{h_{2}}}\|_{V}.
 \end{eqnarray*}
Thus (\ref{eq3.74}) holds.
 \end{proof}


\begin{remark}\label{r10}
The Lemma {\ref{l2.3h}} gives as a first additional information
that,
 for all $g\leq 0$ in $\Omega$ and all $h>0$, the sequence $(u_{g_{h}})$
is increasing and bounded exceptionally, so it is convergent in some
space. We study, in the next sections, the optimal control problems
associated to the variational inequalities ({\ref{iv1}}) and
({\ref{iv2}}) and  the convergence when $h\to +\infty$ in Lemma
{\ref{l6.1}} and Theorem {\ref{th6.1}}  for all $g$, without
restriction to   $g\leq 0$ in $\Omega$.
\end{remark}

\section{Optimal control problems and convergence for
$h\to +\infty$}\label{ocpb} We will first study in this section two
kind of distributed optimal control problems, their existence,
uniqueness results and the relation between them. In fact the
existence and uniqueness, of the solution to the two variational
inequalities (\ref{iv1}) and (\ref{iv2}) allow us to consider
$g\mapsto u_{g}$ and  $g\mapsto u_{gh}$ as a functions from $H$ to
$V$, for any $h>0$.

Let a constant $M>0$. We  define the two cost functional $J : H \to
\br$  and  $J_{h} :  H \to \br$ such that  \cite{JLL}
(see also \cite{KeMu2008}-\cite{KeSa1999})

 \begin{equation}\label{e4.1}
J(g)= {1\over 2}\|u_{g}\|_{H}^{2}+ {M\over 2}\|g\|_{H}^{2},
\end{equation}
\begin{equation}\label{Jh}
 J_{h}(g)= {1\over 2}\|u_{gh}\|_{H}^{2}+ {M\over 2}\|g\|_{H}^{2},
\end{equation}

and we consider the family of distributed optimal control problems
\begin{equation}\label{P}
 \mbox{Find }  g_{op}\in H \quad\mbox{such that} \quad J(g_{op})= \min_{g\in H} J(g),
\end{equation}
\begin{equation}\label{Ph}
 \mbox{Find }  g_{op_{h}}\in H \quad \mbox{such that } \quad J(g_{op_{h}})= \min_{g\in H} J_{h}(g).
\end{equation}

\bigskip
\begin{lemma}\label{l3}
Let $g,\, g_{1},\, g_{2}$ in $H$ and $u_{g},\, u_{g_{1}}, \,
u_{g_{2}}$ are the associated solutions of {(\ref{iv1})}. We  have

\begin{eqnarray}\label{eq3.10}
 \|u_{3}(\mu)-u_{4}(\mu)\|^{2}_{V}
 +\mu(1-\mu)\|u_{g_{1}}-u_{g_{2}}\|^{2}_{V}
+{\mu\over \lambda} I_{14} + {(1-\mu)\over \lambda} I_{24}\nonumber\\
\leq {\mu(1-\mu)\over \lambda^{2}}
\|g_{1}-g_{2}\|^{2}_{H}.
\end{eqnarray}
For $u_{g_{h}},\, u_{g_{1h}}, \, u_{g_{2h}}$  the associated
solutions of {(\ref{iv2})}, we also have
\begin{eqnarray}\label{5.4}
\|u_{4h}(\mu)- u_{3h}(\mu)\|_{V}^{2}
+\mu(1-\mu) \|u_{g_{2h}} -u_{g_{1h}}\|^{2}_{V}
 +{\mu\over \lambda_{h}} I_{14h} + {(1-\mu)\over \lambda_{h}} I_{24h}\nonumber\\ \leq
{\mu(1-\mu)\over \lambda_{h}^{2}}\|g_{1}-g_{2}\|_{H},
\end{eqnarray}
\end{lemma}
\begin{proof}
For $i=1, 2$ we have
$$I_{i4}(\mu) =  a(u_{i}\, , \, u_{4}(\mu) - u_{i})
-( g_{i} , u_{4}(\mu) - u_{i}) +  \int_{\Gamma_{2}} q (
u_{4}(\mu)-u_{i})ds \geq 0$$
and therefore by using Theorem
\ref{th1} and (\ref{L}) we obtain
$$\lambda\|u_{3}(\mu)-u_{4}(\mu)\|^{2}_{V}
+\mu I_{14} + (1-\mu)I_{24}
 \leq \mu(1-\mu) (\alpha + \beta)
\qquad  \forall \mu\in [0 , 1].$$
As
\begin{eqnarray*}
\alpha + \beta &=& a( u_{1} , u_{2}-u_{1}) - (g_{1}  ,  u_{2}-u_{1})
+ \int_{\Gamma_{2}} q ( u_{2}-u_{1})ds
\nonumber\\
 && +  a( u_{2} , u_{1}-u_{2}) - (g_{2}  ,  u_{1}-u_{2})
+ \int_{\Gamma_{2}} q ( u_{1}-u_{2})ds
\nonumber\\
 &\leq&- a(u_{2}- u_{1}  , u_{2}-u_{1}) + (g_{2} -  g_{1}  ,  u_{2}-u_{1})
\nonumber\\
 &\leq& -  \lambda\|u_{2}- u_{1}\|_{V}^{2} + \|g_{2} -  g_{1}\|_{H}\|u_{2}- u_{1}\|_{H}
\nonumber\\
 &\leq& - \lambda \|u_{2}- u_{1}\|_{V}^{2} +{1\over\lambda } \|g_{2} -  g_{1}\|_{H}^{2}
\end{eqnarray*}
thus (\ref{eq3.10}) follows. (\ref{5.4}) follows also from Theorem
\ref{th1} and (\ref{Lh}) as above.
\end{proof}

By using  Lemma \ref{l3} and the references
  \cite{Barbu84a}, \cite{JLL}, we can obtain firstly the existence (not the uniqueness)
  of optimal controls $g_{op}$ and
$g_{op_{h}}$ solution of Problem (\ref{P}) and Problem (\ref{Ph})
respectively. Then, the corresponding uniqueness of the optimal
control problems can be obtained by using (\cite{Mingot1}, pages 166
and 177). Secondly, in order to avoid the use of the conical
differentiability (see \cite{Mingot1}) and by completeness of the
proof of the result we can do another proof of the uniqueness of the
optimal control problems which is not given in \cite{Mingot1}. For
that, we can prove two important equalities (\ref{4.9}) and
(\ref{4.9h}) which allow us to get that $J$ and $J_{h}$  are
strictly convex applications on $H$, so there exist the unique
solutions $g_{op}$ and $g_{op_{h}}$ in $H$ to  the Problem
{\rm(\ref{P})} and Problem {\rm(\ref{Ph})} respectively. This fact
is also very important for us because it permits us to obtain the
convergence in Theorem {\ref{th6.1}}, our mean result, without using
the adjoint state problem.

\begin{proposition}\label{th4.1}
 Let given $g$ in $H$ and  $h>0$,  there exist unique solutions $g_{op}$ and
$g_{op_{h}}$  in $H$ respectively for the Problems {\rm(\ref{P})}
and {\rm(\ref{Ph})}.
\end{proposition}

\begin{proof}
We  remark first  that  using  Lemma \ref{l3} and (\cite{Barbu84a},
\cite{GT}, \cite{JLL}, \cite{Mingot1}) we can obtain the following
classical results
\begin{equation*}\label{4.6}
 \lim_{\|g\|_{H}\to +\infty}J(g) =+\infty,  \quad\mbox{ and }\quad  \lim_{\|g\|_{H}\to +\infty}J_{h}(g) =+\infty,
\end{equation*}
\begin{equation*}\label{4.7}
 J \mbox{ and  } J_{h} \quad \forall h>0, \mbox{ are lower semi-continuous on } H \mbox{ weak,}
\end{equation*}
so we can deduce the existence, of at least, an optimal control
$g_{op}$  solution of Problem {\rm(\ref{P})} and  respectively
 an optimal control $g_{op_{h}}$ solution of Problem {\rm(\ref{Ph})}.

 The uniqueness of the solutions of Problems {\rm(\ref{P})} and {\rm(\ref{Ph})}
can be also obtained by using (\cite{Mingot1}, pages 166 and 177).
For completeness we will prove that the cost functional $J$ and
$J_{h}$ are strictly convex applications on $H$ which are not given
in \cite{Mingot1}. Let $u=u_{g_{i}}$ and $u_{g_{i}h}$ be
respectively the solution of the variational inequalities
{\rm(\ref{iv1})}  and {\rm(\ref{iv2})} with $g=g_{i}$ for $i=1 , 2$.
We have
\begin{eqnarray*}
\|u_{3}(\mu)\|_{H}^{2} = \mu^{2} \|u_{g_{1}}\|_{H}^{2} + (1-\mu)^{2}\|u_{g_{2}}\|_{H}^{2} + 2 \mu(1-\mu)(u_{g_{1}} , u_{g_{2}})
\end{eqnarray*}
then the following equalities hold
\begin{equation}\label{4.9}
\|u_{3}(\mu)\|_{H}^{2} = \mu \|u_{g_{1}}\|_{H}^{2} +
(1-\mu)\|u_{g_{2}}\|_{H}^{2} -\mu(1-\mu)\|u_{g_{2}}-
u_{g_{1}}\|_{H}^{2},
\end{equation}
\begin{equation}\label{4.9h}
\|u_{3h}(\mu)\|_{H}^{2} = \mu \|u_{g_{1}h}\|_{H}^{2} + (1-\mu)\|u_{g_{2}h}\|_{H}^{2}
-\mu(1-\mu)\|u_{g_{2}h}- u_{g_{1}h}\|_{H}^{2}.
\end{equation}
Let now $\mu\in [0 , 1]$ and   $g_{1}, g_{2} \in H$ so we have
\begin{eqnarray*}
\mu J(g_{1})+ (1-\mu)J(g_{2})- J(g_{3}(\mu))= {\mu \over
2}\|u_{g_{1}}\|_{H}^{2} + {(1-\mu) \over 2}\|u_{g_{2}}\|_{H}^{2}
\nonumber\\
-{1 \over 2}\|u_{4}(\mu)\|_{H}^{2} +{M \over 2}\left\{\mu
\|g_{1}\|_{H}^{2} +(1-\mu)\|g_{2}\|_{H}^{2}
-\|g_{3}(\mu)\|_{H}^{2}\right\}
\end{eqnarray*}
and by using (\ref{4.9}) for $g_{3}(\mu)= \mu g_{1} + (1-\mu)g_{2}$
we obtain
\begin{eqnarray}\label{c311}
\mu J(g_{1})+ (1-\mu)J(g_{2})- J(g_{3}(\mu))=
{1 \over 2}\{\mu\|u_{g_{1}}\|_{H}^{2}  + (1-\mu)\|u_{g_{2}}\|_{H}^{2}
-\|u_{4}(\mu)\|_{H}^{2}\}
\nonumber\\
+{M \over 2}\mu(1-\mu) \|g_{1}-g_{2}\|_{H}^{2}.
\end{eqnarray}
Following \cite{Mingot1} we obtain the cornerstone monotony property
\begin{equation}\label{eq3.4}
 u_{4}(\mu) \leq u_{3}(\mu)\quad in \quad \Omega,
    \quad \forall \mu\in [0 , 1],
\end{equation}
and as $u_{4}(\mu)\in K$ so $u_{4}(\mu) \geq 0$ in $\Omega$
for all $\mu\in [0 , 1]$,  we deduce
\begin{eqnarray*}\label{n1}
 \|u_{4}(\mu)\|_{H}^{2}\leq \|u_{3}(\mu)\|_{H}^{2}, \quad \forall \mu\in [0 ,
 1].
\end{eqnarray*}
By using (\ref{4.9}) we have
$$\mu\|u_{g_{1}}\|_{H}^{2}  + (1-\mu)\|u_{g_{2}}\|_{H}^{2}
-\|u_{4}(\mu)\|_{H}^{2}=
\|u_{3}(\mu)\|_{H}^{2}-\|u_{4}(\mu)\|_{H}^{2}+ \mu(1-\mu)
\|u_{g_{1}}-u_{g_{2}}\|_{H}^{2}$$ which is positive for all $\mu\in
[0 , 1]$. Finally we deduce from (\ref{c311}) that
\begin{eqnarray}\label{n}
\mu J(g_{1})+ (1-\mu)J(g_{2})- J(g_{3})
\geq {\mu(1-\mu) \over 2}\left\{\|u_{g_{1}}-u_{g_{2}}\|_{V}^{2} + M \|g_{1}-g_{2}\|_{H}^{2}\right\}
> 0
\end{eqnarray}
for all $\mu\in ]0 , 1[$ and  for all $g_{1}, g_{2}$ in $H$. So $J$ is a strictly convex functional,
 thus the uniqueness of the optimal control for the Problem (\ref{P}) holds.

The uniqueness of the optimal control of the Problem (\ref{Ph})
follows using the analogous inequalities (\ref{c311})-(\ref{n}) for
any $h>0$, that is
\begin{eqnarray}\label{c311h}
\mu J_{h}(g_{1})+ (1-\mu)J_{h}(g_{2})- J_{h}(g_{3}(\mu))&=& {1 \over
2}\{\mu\|u_{g_{1}h}\|_{H}^{2}  + (1-\mu)\|u_{g_{2}h}\|_{H}^{2}
-\|u_{4h}(\mu)\|_{H}^{2}\}
\nonumber\\
&&+{M \over 2}\mu(1-\mu) \|g_{1}-g_{2}\|_{H}^{2}
\end{eqnarray}
from
\begin{equation}\label{eq3.4h}
 u_{4h}(\mu) \leq u_{3h}(\mu)\quad in \quad \Omega,
\end{equation}
so we get
\begin{equation}\label{n1h}
 \|u_{4h}(\mu)\|_{H}^{2}\leq \|u_{3h}(\mu)\|_{H}^{2},
\end{equation}
and obtain
\begin{eqnarray*}\label{nh}
\mu J_{h}(g_{1})+ (1-\mu)J_{h}(g_{2})- J_{h}(g_{3}) \geq {\mu(1-\mu)
\over 2}\left\{\|u_{g_{1h}}-u_{g_{2h}}\|_{V}^{2} + M
\|g_{1}-g_{2}\|_{H}^{2}\right\}
> 0
\end{eqnarray*}
for all $\mu\in ]0 , 1[$, for all $h>0$ and for all $g_{1}, g_{2}$
in $H$. So $J_{h}$ is also a strictly convex functional, thus the
uniqueness of the optimal control for the Problem (\ref{Ph}) holds.
\end{proof}

\begin{remark}\label{r100}
The Proposition {\rm {\ref{th4.1}}} is automatically true (and then
it is not necessary in order to study the convergence given in {\rm
Theorem {\ref{th6.1}})} when the equivalence {\rm(\ref{u3})} is
verified for all $g_{1}, g_{2}$ in $H$.
\end{remark}

Now we study the convergence of  the state $u_{{g_{op}}_{h}h}$, and
the optimal control ${g_{op}}_{h}$, when the heat transfer
coefficient $h$ on $\Gamma_{1}$,
 goes to infinity.  For a given fixed $g\in H$, we have the
 following property which generalizes the one obtained for
 variational equality in \cite{T,TT}. After that, we can study the
 limit $h\to +\infty$ for the general optimal control problems.

\begin{lemma}\label{l6.1}  Let $u_{g_{h}}$ the unique solution of
the variational inequality
 {\rm(\ref{iv2})} and $u_{g}$ the unique solution of the variational inequality {\rm(\ref{iv1})}, then
 $$u_{g_{h}}\to u_{g}\quad in \quad V  \mbox{ strongly  as } h\to +\infty \quad \forall g\in  H.$$
\end{lemma}
\begin{proof}
We take  $v=u_{g}$ in (\ref{iv2}) where $u= u_{g_{h}}$,  recalling that $u_{g}= b$ on $\Gamma_{1}$ and  $h> 1$, we obtain
\begin{eqnarray}\label{q4.1}
&&a_{1}(u_{g_{h}} - u_{g}, u_{g_{h}}- u_{g})  + (h-1)\int_{\Gamma_{1}}(u_{g_{h}} - u_{g})^{2} ds \nonumber\\
&\leq& ( g  ,u_{g_{h}}- u_{g})
     - \int_{\Gamma_{2}} q  (u_{g_{h}}-u_{g})ds
+ \int_{\Gamma_{1}} b  (u_{g_{h}}-u_{g})ds - a_{1}(u_{g}, u_{g_{h}}- u_{g})
\nonumber\\
&\leq& (g  ,u_{g_{h}}- u_{g})
     - \int_{\Gamma_{2}} q  (u_{g_{h}}-u_{g})ds
- a(u_{g}, u_{g_{h}}- u_{g}).
\end{eqnarray}
From what we deduce
that $\|u_{g_{h}}-u_{g}\|_{V}$ and $(h-1)\|u_{g_{h}}-u_{g}\|_{L^{2}(\Gamma_{1})}$ are bounded for all $h>1$. So there exists $\eta\in V$ such that $u_{g_{h}}\tow \eta$ weakly in $V$
 and $\eta \in K$. From (\ref{iv2}) we have also
\begin{equation*}
a(u_{g_{h}} , v- u_{g_{h}}) +  h\int_{\Gamma_{1}} (u_{g_{h}}- b) (v- u_{g_{h}})ds  \geq ( g  , v - u_{g_{h}}) - \int_{\Gamma_{2}} q ( v- u_{g_{h}})ds \quad \forall v\in K_{+},
\end{equation*}
taking $v\in K$ so $v= b$ on $\Gamma_{1}$, thus
\begin{equation}\label{eq6.1}
a(u_{g_{h}} ,  u_{g_{h}})   \leq a(u_{g_{h}} , v)- ( g  , v - u_{g_{h}}) + \int_{\Gamma_{2}} q ( v- u_{g_{h}})ds \quad \forall v\in K.
\end{equation}
Thus we can pass to the limit in (\ref{eq6.1}), for $h\to +\infty$,  to obtain
\begin{equation*}
a(\eta , v- \eta) \geq ( g  , v - \eta) - \int_{\Gamma_{2}} q ( v-\eta)ds \quad \forall v\in K.
\end{equation*}
Using the uniqueness of the solution of (\ref{iv1}) we get that $\eta= u_{g}$.

\bigskip To prove the strong convergence of $u_{g_{h}}$ to $u_{g}$,  when $h\to +\infty$, it is sufficient to use the inequality (\ref{q4.1}) and the weak convergence of $u_{g_{h}}$ to $\eta=u_{g}$ for all $g\in H$.
This ends the proof.
\end{proof}

\bigskip
We give now the main result of the paper which generalizes, for
optimal control problems governed by elliptic variational
inequalities, the convergence result obtained in \cite{GT}.
Moreover, this convergence is obtained without need of
 the adjoint states. We remark here the double dependence on the parameter $h$ in
 the expression of state of the system $u_{{g_{op}}_{h}h}$ corresponding to the optimal control
 ${g_{op}}_{h}$.

\begin{theorem}\label{th6.1}
Let $u_{{g_{op}}_{h}h}$, ${g_{op}}_{h}$  and $u_{g_{op}}$, $g_{op}$
are the states and the optimal controls defined in the problems
{\rm(\ref{Ph})} and {\rm(\ref{P})} respectively. Then, we obtain the
following asymptotic behavior:
\begin{equation}\label{6.1}
 \lim_{h\to +\infty}\|u_{{g_{op}}_{h}h}-u_{g_{op}}\|_{V}= 0.
\end{equation}
\begin{equation}\label{6.2}
 \lim_{h\to +\infty}\|{g_{op}}_{h}-g_{op}\|_{H}= 0.
\end{equation}
\end{theorem}
\begin{proof}
We have first
\begin{eqnarray*}
 J_{h}(g_{op_{h}})= {1\over 2}\|u_{g_{op_{h}h}}\|_{H}^{2} + {M\over 2}\|g_{op_{h}}\|_{H}^{2}
\leq {1\over 2}\|u_{g_{h}}\|_{H}^{2} + {M\over 2}\|g\|_{H}^{2},
\quad \forall g\in H
\end{eqnarray*}
then for $g=0\in H$ we obtain that
\begin{eqnarray}\label{e6.3}
J_{h}(g_{op_{h}})= {1\over 2}\|u_{g_{op_{h}h}}\|_{H}^{2} + {M\over
2}\|g_{op_{h}}\|_{H}^{2}\leq {1\over 2}\|u_{0_{h}}\|_{H}^{2}
\end{eqnarray}
where $u_{0_{h}}\in K_{+}$ is solution of the following elliptic
variational inequality
$$ a_{h}( u_{0_{h}} , v- u_{0_{h}}) \geq -\int_{\Gamma_{2}} q(v- u_{0_{h}})ds + h\int_{\Gamma_{1}} b(v- u_{0_{h}})ds \qquad \forall v\in K_{+}. $$
Taking $v= B$ with $B\in K_{+}$ such that $B=b$ on $\Gamma_{1}$, we get
\begin{eqnarray*}
 a_{1}( u_{0_{h}} , u_{0_{h}}) + (h-1) \int_{\Gamma_{1}}  (u_{0_{h}}-b)^{2} ds
\leq a_{1}( u_{0_{h}} , B) +\int_{\Gamma_{2}} q(B- u_{0_{h}})ds
+ \int_{\Gamma_{1}} b(u_{0_{h}}- b)ds
\end{eqnarray*}
thus   $\|u_{0_{h}}\|_{V}$ is bounded independently of $h$, then
from $\|u_{0_{h}}\|_{H}\leq \|u_{0_{h}}\|_{V}$, we deduce that
$\|u_{0_{h}}\|_{H}$ is bounded independently of $h$. So we deduce
 with (\ref{e6.3}) that $\|u_{g_{op_{h}h}}\|_{H}$ and $\|g_{op_{h}}\|_{H}$ are also bounded
 independently of $h$. So there exists $f$ and $\xi$ in  $H$ such that
\begin{eqnarray}\label{6.5}
 g_{op_{h}} \tow  f \quad in \quad H \quad (weak)\qquad {\rm and }\quad u_{g_{op_{h}h}}\tow\xi \quad in \quad H \quad (weak).
\end{eqnarray}
Taking  now $v=u_{g_{op}}\in K\subset K_{+}$ in  (\ref{iv2}) with $u= u_{g_{op_{h}}h}$ and $g=g_{op_{h}}$, we obtain
\begin{eqnarray*}
 a_{1}( u_{g_{op_{h}}h} , u_{g_{op}} - u_{g_{op_{h}}h})
+(h-1) \int_{\Gamma_{1}}u_{g_{op_{h}}h}(u_{g_{op}} - u_{g_{op_{h}}h})ds
\geq ( g_{op_{h}} , u_{g_{op}} - u_{g_{op_{h}}h}) \nonumber\\- \int_{\Gamma_{2}} q (u_{g_{op}} - u_{g_{op_{h}}h})ds
+
h\int_{\Gamma_{1}} b (u_{g_{op}} - u_{g_{op_{h}}h})ds
\end{eqnarray*}
as $ u_{g_{op}}= b$ on $\Gamma_{1}$ we obtain
\begin{eqnarray*}
 a_{1}( u_{g_{op_{h}}h} -  u_{g_{op}} , u_{g_{op}} - u_{g_{op_{h}}h})
-(h-1) \int_{\Gamma_{1}}(u_{g_{op_{h}}h}- b)^{2}ds
\geq ( g_{op_{h}} , u_{g_{op}} - u_{g_{op_{h}}h})
\nonumber\\- \int_{\Gamma_{2}} q (u_{g_{op}} - u_{g_{op_{h}}h})ds
+
\int_{\Gamma_{1}} b (b - u_{g_{op_{h}}h})ds - a_{1}(u_{g_{op}} , u_{g_{op}} - u_{g_{op_{h}}h})
\end{eqnarray*}
so
\begin{eqnarray*}
a_{1}( u_{g_{op_{h}}h} -u_{g_{op}} ,  u_{g_{op_{h}}h} -u_{g_{op}}  )
+ (h-1)\int_{\Gamma_{1}} (u_{g_{op_{h}}h}- b)^{2}ds \leq
\nonumber\\
\leq ( g_{op_{h}} ,   u_{g_{op_{h}}h} - u_{g_{op}} ) -
 \int_{\Gamma_{2}} q (  u_{g_{op_{h}}h}-   u_{g_{op}})ds
-a( u_{g_{op}} ,  u_{g_{op_{h}}h} -u_{g_{op}})
\end{eqnarray*}
thus there exists a constant $C>0$ which does not depend on $h$ such
that (as $h\to +\infty$ we can take $h>1$):
\begin{eqnarray*}\label{6.q}
\|u_{g_{op_{h}}h} -u_{g_{op}}\|_{V}\leq C
\quad \mbox{  and   } \quad    (h-1)\int_{\Gamma_{1}} |u_{g_{op_{h}}h}- b|^{2}ds \leq C,
\end{eqnarray*}
then
\begin{eqnarray}\label{6.6}
u_{g_{op_{h}}h} \tow \xi \quad in \quad V \quad weak \quad \mbox{(in
H strong),}
\end{eqnarray}
\begin{eqnarray}\label{6.7}
 u_{g_{op_{h}}h} \to  b\quad in \quad L^{2}(\Gamma_{1}) \quad strong,
\end{eqnarray}
and then $\xi\in K$.

Now taking $v\in K$  in (\ref{iv2}) where $u= u_{g_{op_{h}}h}$ and $g=g_{op_{h}}$ so
 \begin{eqnarray*}\label{}
  a_{h}(u_{g_{op_{h}}h} , v- u_{g_{op_{h}}h}) \geq ( g_{op_{h}} , v- u_{g_{op_{h}}h}) - \int_{\Gamma_{2}}q(v- u_{g_{op_{h}}h})ds
+ h\int_{\Gamma_{1}} b(v- u_{g_{op_{h}}h})ds
 \end{eqnarray*}
as $v\in K$ so $v=b$ on $\Gamma_{1}$, thus we obtain

 \begin{eqnarray*}\label{}
  a(u_{g_{op_{h}}h} , u_{g_{op_{h}}h})
+ h \int_{\Gamma_{1}}(u_{g_{op_{h}}h}-b)^{2}ds
\leq a( u_{g_{op_{h}}h} ,  v) -
( g_{op_{h}}  ,  v- u_{g_{op_{h}}h})
 \nonumber\\+ \int_{\Gamma_{2}}q(v- u_{g_{op_{h}}h})ds.
 \end{eqnarray*}
Thus
\begin{eqnarray*}\label{}
  a(u_{g_{op_{h}}h} , u_{g_{op_{h}}h})
\leq a( u_{g_{op_{h}}h} ,  v) -
( g_{op_{h}}  ,  v- u_{g_{op_{h}}h})
 + \int_{\Gamma_{2}}q(v- u_{g_{op_{h}}h})ds,
 \end{eqnarray*}
using (\ref{6.5}) and (\ref{6.6}) we deduce that
\begin{eqnarray*}\label{}
  a(\xi , v- \xi) \geq  (f ,  v -\xi)
 - \int_{\Gamma_{2}}q(v -\xi)ds, \quad \forall v\in K,
 \end{eqnarray*}
so by the uniqueness of the solution of the variational inequality
(\ref{iv1}) we obtain that
\begin{eqnarray}\label{xi}
u_{f}= \xi.
 \end{eqnarray}
Now we prove that  $f= g_{op}$.  Indeed we have
 \begin{eqnarray*}\label{}
  J(f)&=&{1\over 2} \|\xi\|_{H}^{2} + {M\over 2} \|f\|_{H}^{2}
\nonumber\\
&\leq& \liminf_{h\to +\infty} \left\{{1\over 2}
\|u_{g_{op_{h}}h}\|_{H}^{2} + {M\over 2} \|g_{op_{h}}\|_{H}^{2}
\right\} =\liminf_{h\to +\infty} J_{h}(g_{op_{h}})
\nonumber\\
&\leq& \liminf_{h\to +\infty} J_{h}(g) =\liminf_{h\to +\infty}
\left\{{1\over 2} \|u_{g_{h}}\|_{H}^{2} + {M\over 2} \|g\|_{H}^{2}
\right\}
 \end{eqnarray*}
using now the strong convergence $u_{g_{h}}\to u_{g}$ as $h\to
+\infty,\; \forall \; g\in H$ (see Lemma \ref{l6.1}), we obtain that
 \begin{eqnarray}\label{6.9}
J(f)\leq \liminf_{h\to +\infty} J_{h}(g_{op_{h}}) \leq
 {1\over 2} \|u_{g}\|_{H}^{2} + {M\over 2} \|g\|_{H}^{2}= J(g),
\qquad \forall g\in H
 \end{eqnarray}
then by the uniqueness of the optimal control problem (\ref{P}) we get
 \begin{eqnarray}\label{f}
f= g_{op}.
 \end{eqnarray}

Now we prove the  strong convergence of  $u_{g_{op_{h}}h}$ to $\xi$ in $V$, indeed
taking $v=\xi$ in (\ref{iv2}) where $u=u_{g_{op_{h}}h}$ and $g= g_{op_{h}}$ we get
 \begin{eqnarray*}
 a_{h}( u_{g_{op_{h}}h} , \xi -u_{g_{op_{h}}h}) \geq (g_{op_{h}} , \xi -u_{g_{op_{h}}h}) - \int_{\Gamma_{2}} q (\xi -u_{g_{op_{h}}h}) ds
+ h \int_{\Gamma_{1}} b(\xi -u_{g_{op_{h}}h})ds,
 \end{eqnarray*}
as $\xi\in K$ so $\xi=b$ on $\Gamma_{1}$, we obtain
 \begin{eqnarray*}\label{}
 a_{1}( u_{g_{op_{h}}h} -\xi , u_{g_{op_{h}}h}-\xi)
+ (h-1)  \int_{\Gamma_{1}}( u_{g_{op_{h}}h} -\xi)^{2}ds
\leq (g_{op_{h}} , u_{g_{op_{h}}h}-\xi)
\nonumber\\
 + \int_{\Gamma_{2}} q (\xi -u_{g_{op_{h}}h}) ds
+ a(\xi , \xi - u_{g_{op_{h}}h})
 \end{eqnarray*}
thus
\begin{eqnarray*}
\lambda_{1}\|u_{g_{op_{h}}h} -\xi\|_{V}^{2}
\leq (g_{op_{h}} , u_{g_{op_{h}}h}-\xi)
 + \int_{\Gamma_{2}} q (\xi -u_{g_{op_{h}}h}) ds
+ a(\xi , \xi - u_{g_{op_{h}}h}).
 \end{eqnarray*}
Using (\ref{6.6}) we deduce that
\begin{eqnarray*}\label{}
\lim_{h\to +\infty}\|u_{g_{op_{h}}h} -\xi\|_{V} = 0,
 \end{eqnarray*}
and with (\ref{xi}) we deduce (\ref{6.1}).
Moreover, as $f\in H$, then from  (\ref{6.9}) with $g=f$ and (\ref{f}) we can write
\begin{eqnarray}\label{eq6.11}
 J(f)&=& J(g_{op}) ={1\over 2}\|u_{g_{op}}\|_{H}^{2} + {M\over 2}\|g_{op}\|_{H}^{2}
\nonumber\\
&=& \lim_{h\to+\infty} J_{h}(g_{op_{h}}) = \lim_{h\to+\infty}
\left\{{1\over 2}\|u_{g_{op_{h}}h}\|_{H}^{2} + {M\over
2}\|g_{{op_{h}}}\|_{H}^{2}\right\}
\end{eqnarray}
and using (\ref{6.1}) the strong convergence $u_{g_{op_{h}}h}\to \xi=u_{f}=u_{g_{op}}$ in $V$, we get
\begin{eqnarray}\label{eq6.12}
\lim_{h\to+\infty}\|u_{g_{op_{h}}h}\|_{H} = \|u_{g_{op}}\|_{H},
\end{eqnarray}
thus from (\ref{eq6.11}) and  (\ref{eq6.12}) we get
\begin{eqnarray}\label{eq6.13}
\lim_{h\to+\infty}\|g_{{op_{h}}}\|_{H}= \|g_{op}\|_{H}.
\end{eqnarray}
Finally
\begin{eqnarray}\label{eq6.14}
\lim_{h\to +\infty}\|g_{{op_{h}}}- g_{op}\|_{H}^{2} = \lim_{h\to
+\infty}\left( \|g_{{op_{h}}}\|_{H}^{2}+ \|g_{op}\|_{H}^{2}
-2(g_{{op_{h}}} , g_{op})\right).
\end{eqnarray}
By  the first part of (\ref{6.5}) we obtain that
$$\lim_{h\to +\infty}\left(g_{{op_{h}}} , g_{op}\right) = \|g_{op}\|_{H}^{2},$$
so from (\ref{eq6.13}) and (\ref{eq6.14}) we  get (\ref{6.2}). This ends  the proof.
\end{proof}

\begin{remark}\label{rf}
Much of the recent literature on optimal control problems governed
by variational inequalities (often called mathematical programs with
equilibrium constraints (MPEC)) is focused on the numerical
realization of stationary points to these problems. See for example
recent works as e.g. {\rm\cite{H2009}} and their references within
it. The numerical analysis of the convergence of optimal control
problems governed by elliptic variational equalities {\rm\cite{GT}}
is given in {\rm\cite{T2011}} but the numerical analysis of the
corresponding convergence of optimal control problems governed by
elliptic variational inequalities given by {\rm Theorem \ref{th6.1}}
is an open problem.
\end{remark}

\bigskip

\noindent{\bf Conclusions:} In this paper we have first established
the error estimate between the convex combination
$u_{3}(\mu)= \mu u_{g_{1}}+ (1-\mu)u_{g_{2}}$
 of two solutions
$u_{g_{1}}$ and $u_{g_{2}}$ for elliptic variational inequality
corresponding to the data $g_{1}$ and $g_{2}$ respectively, and the
solution $u_{4}(\mu)= u_{g_{3}(\mu)}$  of the same elliptic
variational inequality corresponding to the convex combination
$g_{3}(\mu)= \mu g_{1} + (1-\mu)g_{2}$ of the two data. This result
complements and generalizes the previous one given in \cite{MB-DT1}.

Using the existence and uniqueness of the solution to particular
elliptic variational inequality, we consider a family of distributed
optimal control problems on the internal energy $g$ associated to
the heat transfer coefficient $h$ defined on a portion of the
boundary of the domain.  Using the monotony property \cite{Mingot1}
(see (\ref{eq3.4}) and (\ref{eq3.4h}) ) we can obtain the strict
convexity of the cost functional (\ref{e4.1}) and (\ref{Jh}), and
the existence and uniqueness of the distributed optimal control
problems (\ref{P}) and (\ref{Ph}) for any $h>0$ holds by a different
way used in \cite{Mingot1} avoiding the conical differentiability of
the cost functional. Then  we prove that the optimal control
$g_{op_{h}}$ and its corresponding state of the system
$u_{g_{op_{h}h}}$ are strongly convergent, when $h\to +\infty$, to
$g_{op}$ and $u_{g_{op}}$ which are respectively the optimal control
and its corresponding state of the system, for a limit Dirichlet
distributed optimal control problems. We obtain our results without
using the notion of adjoint state (i.e. the Mignot's conical
differentiability) of the optimal control problems which is a very
important advantage with respect to the previous result given in
\cite{GT} for elliptic variational equalities.

\bigskip
\noindent{\bf Acknowledgements:} This work was realized while the
second author was a visitor at Saint Etienne University (France) and
he is grateful to this institution for its hospitality, and it was partially supported by Grant FA9550-1061-0023. We would
like to thank two anonymous referees for their constructive comments
which improved the readability of the manuscript.



\begin{thebibliography}{00}
\bibitem{A2006} K. Ait Hadi (2006)
 \emph{Optimal control of the obstacle problem: optimality conditions},
   IMA J. Math. Control Inform.  23, 325-334.

\bibitem{bacuta2003} C. Bacuta, J. H. Bramble and J.E. Pasciak (2003)
\emph{Using finite element tools in proving shift theorems for elliptic boundary value problems},
Numer. Linear Algebra Appl. 10, 33-64.

\bibitem{Barbu84a} V. Barbu  (1984)
\emph{Optimal control of variational inequalities}. Research Notes
in Mathematics, 100. Pitman (Advanced Publishing Program), Boston,
MA.

\bibitem{B1997b} M. Bergounioux (1997)
 \emph{Use of augmented Lagrangian methods for the
optimal control of obstacle problems},
   J. Optim. Th. Appl.  95 (1), 101-126.

\bibitem{Mber2} M. Bergounioux, and K. Kunisch (1997)
\emph{Augmented Lagrangian Techniques for elliptic state constrained
optimal Control of problems}, SIAM J. Control Optim. 35, No.5,
1524-1543.

\bibitem{BM2000} M. Bergounioux and F. Mignot (2000)
 \emph{Optimal control of obstacle problems: existence of Lagrange multipliers},
   ESAIM: Control, Optim. and Calculus of Variations  5, 45-70.

\bibitem{MB-DT1} M. Boukrouche and D. A. Tarzia (2007)
 \emph{On a convex combination of solutions to elliptic variational
inequalities}, Electro. J. Diff. Equations 2007, No. 31, pp. 1-10

\bibitem{Acapatina2000} A. Capatina (2000)
 \emph{Optimal Control of a Signorini contact problem}, Numer. Funct. Anal. and Optimiz. 21(7-8), 817-828

\bibitem{CaJa1959} H.S. Carslaw and  J.C. Jaeger (1959)
\emph{Conduction of heat in solids}. Clarendon Press, Oxford.

\bibitem{GT} C.M. Gariboldi and D.A. Tarzia (2003)
\emph{Convergence of distributed
optimal controls on the internal energy in mixed
  elliptic problems when the heat transfer coefficient goes to infinity},
   Appl. Math. Optim.  47 (3),  213-230.

\bibitem{gri85} P. Grisvard (1985)
\emph{Elliptic problems in non-smooth domains}, Pitman, London.

\bibitem{HMRS2009} R. Haller-Dintelmann, C. Meyer, J. Rehberg and A. Schiela (2009)
\emph{Holder continuity and optimal control for nonsmooth elliptic
problems}, Appl. Math. Optim. 60, 397-428.

\bibitem{Jhas1986} J. Haslinger and T. Roubicek  (1987)
\emph{Optimal control of  variational inequalities. Approximation
Theory and Numerical Realization}, Appl. Math. Optim. 14,  187-201.

\bibitem{H2001} M. Hinterm\"uller (2001)
\emph{Inverse coefficient problems for variational inequalities:
Optimality and numerical realization}, Math. Modelling Numer. Anal.
35, 129-152.

\bibitem{H2008} M. Hinterm\"uller (2008)
\emph{An active-set equality constrained Newton solver with
feasibility restoration for inverse coefficient problems in elliptic
variational inequalities}, Inverse Problems 24 (Article 034017),
1-23.

\bibitem{H2009} M. Hinterm\"uller and I. Kopacka (2009)
\emph{Mathematical programs with complementary constraints in
fucntion space: C- and strong stationarity and a path-following
algoritm}, SIAM J. Optim. 20, 868-902.

\bibitem{IK2000} K. Ito and K. Kunisch (2000)
\emph{Optimal control of elliptic variational inequalities},
   Appl. Math. Optim.  41, 343-364.

\bibitem{IK2008} K. Ito and K. Kunisch (2008)
\emph{Lagrange multiplier approach to variational problems and
applications}, SIAM, Philadelphia.

\bibitem{KeMu2008}
S. Kesavan and T. Muthukumar (2008) \emph{Low-cost control problems
on perforated and non-perforated domains},
 Proc. Indian Acad. Sci. (Math. Sci.) 118, No. 1, 133-157.

\bibitem{KeSa1997}
S. Kesavan and J. Saint Jean Paulin (1997) \emph{Homogenization of
an optimal control problem}, SIAM J. Control Optim. 35, No.5,
1557-1573.

\bibitem{KeSa1999}
S. Kesavan and J. Saint Jean Paulin (1997) \emph{Optimal control on
perforated domains}, J. Math. Anal. Appl. 229, 563-586.

\bibitem{K2008}
S. D. Kim (2008) \emph{Uzawa algorithms for coupled Stokes equations
from the optimal control problem}, Calcolo 46, 37-47.

\bibitem{Kind80}
D. Kinderlehrer and G. Stampacchia (1980)
\emph{An introduction to variational inequalities and their applications}.
Academic Press, New York.

\bibitem{LCB2008}
L. Lanzani, L. Capogna and R.M. Brown (2008) \emph{The mixed problem
in $L^{p}$ for some two-dimensional Lipschitz domains},
 Math. Annalen 342, 91-124.

\bibitem{JLL} J.L. Lions (1968)
\emph{Contr\^ole optimal de syst\`emes gouvern\'es par des
\'equations aux d\'eriv\'ees partielles}, Dunod, Paris.

\bibitem{JLL1}
J. L. Lions and G. Stampacchia (1967)
\emph{Variational inequalities},
 Comm. Pure Appl. Math. 20, 493-519.

\bibitem{MT} J.L. Menaldi and D. A. Tarzia (2007)
 \emph{A distributed parabolic control with mixed boundary
 conditions}, Asymptotic Anal. 52, 227-241.

\bibitem{MRT2006} C. Meyer, A. R\"osch and F. Tr\"oltzsch (2006)
\emph{Optimal control pf PDEs with regularized pointwise state
constraints}, Comput. Optim. Appl. 33, 209-228.

\bibitem{Mingot1} F. Mignot (1976)
\emph{Contr\^ole dans les in\'equations variationelles elliptiques},
 J. Functional Anal.  22, No. 2, 130-185.

\bibitem{MP1984} F. Mignot and J.P. Puel (1984)
 \emph{Optimal control in some variational inequalities},
   SIAM J. Control Optim.  22 (3), 466-476.


\bibitem{NPS2006} P. Neittaanm\"aki, J. Sprekels and  D. Tiba (2006)
\emph{Optimization of elliptic systems. Theory and applications},
Springer Monographs in Mathematics, Springer, New York.



\bibitem{Pa1977} F. Patrone (1977)
 \emph{On the optimal control of variational inequalities},
   J. Optim. Th. Appl.  22 (3), 373-388.

\bibitem{R1987} J.F. Rodrigues (1987)
\emph{Obstacle problems in mathematical physics}, North-Holland,
Amsterdam.

\bibitem{Stamp64}
 G. Stampacchia (1964)
\emph{Formes bilin\'eaires coercitives sur les ensembles convexes,}
 C. R. Acad. Sci. Paris.  258,   4413-4416.

\bibitem{TT} E.D. Tabacman and D. A. Tarzia (1989)
\emph{Sufficient and/or necessary condition for the heat transfer
coefficient on $\Gamma_{1}$ and the heat flux on $\Gamma_{2}$ to
obtain a steady-state two-phase Stefan problem},
 J. Diff. Equations 77, No.1, 16-37.

\bibitem{T} D. A. Tarzia (1979)
\emph{Una familia de .problemas que converge hacia el caso
estacionario del problema de Stefan a dos fases},
 Math. Notae 27, 157-165.

\bibitem{T2011} D. A. Tarzia (2011)
\emph{Convergence of a family of distributed discrete elliptic
optimal control problems with respect to a parameter}, in ICIAM
2011, Vancouver, July 18-22, 2011.

\bibitem{T2010} F. Tr\"oltzsch (2010)
\emph{Optimal control of partial differential equations: Theory,
methods and applications}, American Math. Soc., Providence.

\bibitem{YC2004} Y. Ye and Q. Chen (2004)
 \emph{Optimal control of the obstacle problem in a quasilinear elliptic variational inequality},
   J. Math. Anal. Appl.  294,  258-272.

\end{thebibliography}
\end{document}